\documentclass[12pt, a4paper]{article}
\usepackage{titlesec} 
\usepackage{amsfonts}
\usepackage{amsthm}
\usepackage{amssymb}
\usepackage{amsmath}
\usepackage{theoremref}
\usepackage{enumerate}
\usepackage{tikz}
\usepackage{bbm}
\usepackage{authblk}
\usepackage{commath}
\usepackage{mathtools}
\usepackage{xcolor}

\makeatletter
\g@addto@macro\bfseries{\boldmath}
\makeatother

\newcommand{\C}{\mathbb{C}}
\newcommand{\R}{\mathbb{R}}

\newcommand{\e}{\varepsilon}
\newcommand{\z}{\zeta}

\newcommand{\la}{\lambda}
\newcommand{\fr}{\frac}
\newcommand{\conj}[1]{\overline{#1}}
\newcommand{\D}{\mathbb{D}}
\newcommand{\B}{\mathcal{B}}

\newcommand{\Dy}{\mathcal{D}}

\renewcommand\Re{\operatorname{Re}}

\newtheorem{thm}{Theorem}[section]

\newtheorem{lemma}[thm]{Lemma}
\newtheorem{cor}[thm]{Corollary}
\newtheorem{prop}[thm]{Proposition}
\theoremstyle{definition}

\theoremstyle{definition}

\newtheorem*{problem*}{Problem}

\titleformat{\section}
{\normalfont\scshape\centering}{\thesection}{1em}{}

\titleformat{\subsection}
[runin]{\bfseries\scshape}{\thesubsection}{0.5em}{}

\titleformat{\subsubsection}
[runin]{}{\thesubsubsection}{0.5em}{}

\begin{document}
\title{\textbf{Bloch functions and Bekoll\'e-Bonami weights}.}
\author{Adem Limani \& Artur Nicolau 
\thanks{The second author is supported in part by the Generalitat de Catalunya (grant 2017 SGR 395) and the Spanish Ministerio de Ciencia e Innovaci\'on (project  MTM2017-85666-P).} }
\date{}
\newcommand{\Addresses}{{
  \bigskip
  \footnotesize

  A.~Limani, \textsc{Centre for Mathematical Sciences, Lund University, P.O Box 118, SE-22100, Lund, Sweden}\par\nopagebreak
  \textit{E-mail address}, A.~Limani: \texttt{adem.limani@math.lu.se}

  \medskip

  A.~Nicolau 
  , \textsc{Departament de  
Matem\`atiques, Universitat Aut\`onoma de Barcelona, 08193 Barcelona}\par\nopagebreak
  \textit{E-mail address}, A.~Nicolau: \texttt{artur@mat.uab.cat}

}}
\maketitle

%
%
\begin{abstract}
\noindent We study analogues of well-known relationships between Muckenhoupt weights and $BMO$ in the setting of Bekoll\'e-Bonami weights. For Bekoll\'e-Bonami weights of bounded hyperbolic oscillation, we provide distance formulas of Garnett and Jones-type, in the context of $BMO$ on the unit disc and hyperbolic Lipschitz functions. This leads to a characterization of all weights in this class, for which any power of the weight is a Bekoll\'e-Bonami weight, which in particular reveals an intimate connection between Bekoll\'e-Bonami weights and Bloch functions. On the open problem of characterizing the closure of bounded analytic functions in the Bloch space, we provide a counter-example to a related recent conjecture. This shed light into the difficulty of preserving harmonicity in approximation problems in norms equivalent to the Bloch norm. Finally, we apply our results to study certain spectral properties of Cesar\'o operators. 
\end{abstract}
%
%
\section{Introduction.}
Let $\D$ denote the unit disc and $\partial \D$ the unit circle of the complex plane. Let $L^\infty (\partial \D)$ be the algebra of essentially bounded real-valued functions defined on the unit circle and let $BMO (\partial \D) $ be the space of integrable functions $f: \partial \D \rightarrow \R$ such that 
\begin{equation*}
    \|f \|_{BMO (\partial \D)}= \sup_I \frac{1}{|I|} \int_I |f - f_I| dm < \infty, 
\end{equation*}
where the supremum is taken over all arcs $I \subset \partial \D$, $m$ denotes the Lebesgue measure on $\partial \D$ and $f_I$ denotes the mean of $f$ over $I$. A positive integrable function $w$ on the unit circle is called an $A_2$ weight, if 
\begin{equation*}
    \sup_I \left(\frac{1}{m(I)} \int_I w dm \right) \left(\frac{1}{m(I)} \int_I w^{-1} dm \right) < \infty,
\end{equation*}
where the supremum is taken over all arcs $I \subset \partial \mathbb{D} $. Functions  satisfying this condition are called Muckenhoupt $A_{2}$ weights and appear naturally when studying the boundedness of Calder\'on-Zygmund operators on weighted $L^2$ spaces. Muckenhoupt weights are also intimately related to $BMO(\partial \D)$ functions. The well-known equivalence between the $A_2$ condition and the Helson-Szeg\"o Theorem gives the following beautiful  distance formula. For  $f \in BMO (\partial \D)$ consider the quantity $\alpha (f) = \inf \{t >0 : e^{f / t } \in A_2 \}. $ 
Then there exists a universal constant $C>0$ such that for any $f \in BMO (\partial \D)$, we have 
\begin{equation}\label{distance}
    C^{-1} \alpha (f) \leq \inf \{ \|f - h \|_{BMO (\partial \D)} : h \in L^{\infty} (\partial \D) \} \leq C \alpha (f). 
\end{equation}
In particular, a function $f \in BMO (\partial \D)$ belongs to the closure of $L^\infty (\partial \D)$ in $BMO (\partial \D)$ if and only if $e^{\lambda f } \in A_2$, for any $\lambda >0$. See Chapter VI of \cite{Gar}. 
The main purpose of this paper is to treat analogous results for Bekoll\'e-Bonami weights on $\D$. 
\noindent
To this end, we equip the unit disc $\D$ with the hyperbolic metric
\[\beta(z,\z) = \frac{1}{2} \log\left(\frac{1+\, \abs{\phi_{z}(\z)}^2}{1-\, \abs{\phi_{z}(\z)}^2}\right),  \qquad \, z,\zeta \in \D.
\] 
Here $\phi_{z}$ denotes the M\"obius automorphism on $\D$, given by
\[ \phi_{z}(\z) = \fr{z-\z}{1-\conj{z}\z},  \qquad  \, z,\z \in \D.
\] 
Let $dA$ denote the normalized Lebesgue area-measure on $\D$. An integrable positive function $w$ defined on $\D$ is called a $B_{2}$-weight if 
\begin{equation}\label{B2weight} \left[w\right]_{B_{2}}^2 = \sup_{I} \left(\fr{1}{A(Q_{I})} \int_{Q_{I}} w dA \right)\left(\fr{1}{A(Q_{I})}\int_{Q_{I}}w^{-1} dA\right) < \infty , 
\end{equation}
where the supremum is taken over all arcs $I \subset \partial \D$ and $Q_{I}$ denotes the Carleson square associated to $I$, defined by 
\[ Q_{I} = \left\{z\in \D \setminus \{0\}: \frac{z}{|z|} \in I \, \, , \, 1-|z| < m(I)  \right\}.
\]
Such weights naturally appeared when studying the Bergman projection 
\[ P(f)(z) = \int_{\D}\fr{f(\z)}{(1-\overline{\z} z)^{2}}dA(\z), \qquad z\in \D,
\]
on the space $L^{2}(\D,wdA)$, which consists of square-integrable functions on $\D$ with respect to the measure $wdA$. It was proven in 1978, by Bekoll\'e and Bonami that the Bergman projection $P:L^{2}(\D, wdA) \rightarrow L^{2}(\D, wdA)$ is bounded if and only if $w$ is a $B_{2}$-weight. Actually, they introduced the notion of so-called $B_{p}$-weights, when characterizing weights for which the Bergman projection is bounded on $L^{p}(\D,wdA)$, $1<p<\infty$. See \cite{BeBo}. In this sequel, we will mainly restrict our attention to a specific class of weights, that is, a specific class of integrable and positive functions. A weight $w:\D \rightarrow (0,\infty)$ is said to be of \textbf{bounded hyperbolic oscillation} (on hyperbolic discs with fixed radii), if there exists a constant $C(w)>0$, such that
\begin{equation}\label{WCHD} \abs{\log w(z)- \log w(\zeta)} \leq C(w) \, \left(1+  \beta(z,\z) \right), \qquad  z,\zeta \in \D.
\end{equation} 
Notice that condition \eqref{WCHD} is conformally invariant, that is, invariant under M\"obius automorphisms.
This particular class of weights has been recently studied in \cite{AlePottReg}, where the authors showed that within this class of weights, one can develop a fruitful theory of $B_{\infty}$-weights, satisfying the H\"older inequality, similar to the well-studied Muckenhoupt $A_{\infty}$-weights on $\partial \D$. In general, Bekoll\'e-Bonami weights are far more ill-behaved and lack self-improvement properties, see \cite{Bor}. Functions satisfying \eqref{WCHD} have also naturally appeared in \cite{BeCoZh}, when studying symbols of Hankel operators on the Bergman spaces.\\ \\
\noindent
A function $v:\D \rightarrow \C$ will be called a hyperbolic Lipschitz function if
\begin{equation}\label{HypHLip}\|v\|_{\text{HLip}}= \sup_{\substack{z,\z \in \D \\ z\neq \z}} \fr{\abs{v(z)-v(\z)}}{\beta(z,\z)} < \infty,
\end{equation}
that is, if it is Lipschitz continuous when $\D$ and $\C$ are endowed with the hyperbolic metric and the euclidean metric, respectively. We denote the space of hyperbolic Lipschitz functions by $\textup{HLip}(\D)$.
The subspace of holomorphic functions on $\D$ which belong to $\textup{HLip}(\D)$ is precisely the classical Bloch space $\B$, which consists of analytic functions $g$ in $\D$, such that 
\[ \|g\|_{\B}= \sup_{z\in \D} (1-|z|^{2})|g'(z)| < \infty.
\]
Indeed, is not difficult to show that $\|g\|_{\B} = \|g\|_{\text {HLip}}$. In fact, the Bloch space $\B$ is a M\"obius invariant Banach space (modulo constants) and plays a crucial role in the theory of conformal mappings and in the theory of Bergman spaces. See \cite{Pom}, \cite{DurSch} and \cite{HeKoZh} for further details on this matter. By taking the closure of analytic polynomials in $\B$, we obtain the so-called little Bloch space $\B_{0}$, which consists of analytic functions $g\in \B$, satisfying
\[ \lim_{|z|  \rightarrow 1-}(1-|z|^{2})|g'(z)| = 0.
\]
Let $L^\infty ( \D)$ be the algebra of essentially bounded real-valued functions defined on the unit disc and let $BMO(\D)$ denote the space of functions of bounded mean oscillation with respect to the Lebesgue measure restricted to the unit disc. In other words, an integrable real-valued function $f$ on $\D$ belongs to $BMO (\D)$ if 
\begin{equation*}
    \|f\|_{BMO(\D)} = \sup \frac{1}{A(D \cap \D)} \int_{D \cap \D} |f - f_{D \cap \D}| dA < \infty , 
\end{equation*}
where the supremum is taken over all discs $D$ centered at points in $\D$ and $f_{D \cap \D}$ denotes the mean of $f$ over $D \cap \D$. Taking the closure of polynomials in $BMO(\D)$, we obtain the subspace $VMO(\D)$. These spaces have previously been studied in \cite{CoRoWe} by Coifman, Rochberg and Weiss, where it is for instance showed that the space of analytic functions in $\D$ whose real or imaginary parts belong to $BMO(\D)$ (respectively $VMO(\D)$) coincide with $\B$ (respectively $\B_{0}$), with equivalent norms (modulo constants). It is not difficult to see that if the weight $w$ has bounded hyperbolic oscillation, then $\log w \in BMO(\D)$. \\ \\ 
In Lemma \ref{BB-BMO} we shall show that Bekoll\'e-Bonami weights and $BMO(\D)$ are intimately related. In fact, if $w\in B_2$ then 
\[ A(\{z\in Q: |\log w(z) -(\log w)_{Q}|> \la \}) \leq 2e^{2} \left[w\right]_{B_{2}}^{2} e^{-\la} A(Q)
\]
for any Carleson square $Q$ and any $\la > 2+ \log \left[w\right]_{B_{2}}$. This is a uniform John-Nirenberg type estimate which is distinctive of BMO-functions of bounded norm. Conversely, there exists a universal constant $c>0$ such that $e^{cf} \in B_2$, for any $f\in BMO  (\D)$ with $\|f\|_{BMO(\D)}= 1$. Moreover, we shall in Proposition \ref{HBekBlo} illustrate a close connection between $B_{2}$-weights which are exponentials of harmonic functions and the Bloch space. Our first main result is an analogue of formula \eqref{distance}. Given $f \in BMO (\D)$, we set 
\begin{equation}\label{beta}
    \gamma (f) = \inf \{ t >0 : e^{f/ t} \in B_2 \}. 
\end{equation} 

\begin{thm}\label{T1} 
Let $f$ be a real-valued function on $\D$ with the property that the weight $e^{f}$ has bounded hyperbolic oscillation.
\begin{itemize}
\item[(i)]
Then we have 
\begin{equation*}
    2 \gamma (f) \leq \inf \{\|f-h  \|_{\text {HLip}} : h \in L^{\infty} (\D) \} \leq 4 \gamma (f). 
\end{equation*}

\item[(ii)]  There exists a universal constant $C >0$, such that
\begin{equation*}
    C^{-1} \gamma (f) \leq \inf \{\|f-h  \|_{\text BMO (\D)} : h \in L^{\infty} (\D) \} \leq C \gamma (f). 
\end{equation*}
\end{itemize}
\end{thm}
\noindent
Note that in the statement of (i) the function $f$ may not belong to $\text{HLip}(\D)$. In fact,   
(i) of \thmref{T1} says that there exists $h \in L^\infty (\D)$, such that $f-h \in \text{HLip}(\D)$ and 
verifies the estimates in (i). From \thmref{T1} we now deduce.

\begin{cor}\label{closure}
Let $w$ be a weight on $\D$ with bounded hyperbolic oscillation. Then the following statements are equivalent: 
\begin{itemize}
\item[(i)]
For any $\lambda >0$, we have $w^\lambda \in B_2$. 

\item[(ii)] For any $\varepsilon >0$ there exists $h_\varepsilon \in L^\infty (\D)$ such that  
$$\|\log w -h_{\varepsilon} \|_{ \text HLip} < \varepsilon .$$  

\item[(iii)] For any $\varepsilon >0$ there exists $h_\varepsilon \in L^\infty (\D)$ such that  
$$\|\log w -h_{\varepsilon} \|_{ BMO (\D)} < \varepsilon . $$ 

\item[(iv)] For any $\e>0$ there exists $C(\e)>0$ such that 
\begin{equation}\label{Maincond} \abs{\log w(z) - \log w(\zeta) } \leq C(\e) + \e \beta(z, \zeta), \qquad  z,\zeta \in \D. 
\end{equation}

\end{itemize}
\end{cor}

\noindent
A naturally occurring class of weights which have bounded hyperbolic oscillation are weights of the form $e^{\Re(g)}$, for $g\in \B$. These weights appear in the study of spectral properties of generalized Cesar\'o operators on the Bergman space. See \cite{AleCon} and \cite{AdBa}. For our purposes, we highlight the following immediate, yet important Corollary of Theorem \ref{T1}.
 
\begin{cor}\label{clbl}
Let $g\in \B$. Then $e^{\Re( \la g)}$ is a $B_{2}$-weight for all $\lambda \in \C$, if and only if for any $\e>0$, there exists $C(\e)>0$ such that 
\begin{equation}\label{AMaincond} \abs{g(z) - g(\zeta)} \leq C(\e) + \e \beta(z, \z), \qquad \,  z,\zeta \in \D.
\end{equation}
\end{cor}
\noindent Let $H^\infty$ be the algebra of bounded analytic functions  in $\D$. Schwarz's Lemma gives that $H^\infty \subset \B $. Observe that if $g$ belongs to the closure of $H^{\infty}$ in  $\B$, then it satisfies condition (\ref{AMaincond}). 
Indeed, if $g_{\e}\in H^{\infty}$ is an approximate of $g$ in $\B$, then 
 \[ |g(z)-g(\zeta)| 
  \leq 2\|g_{\e}\|_{\infty} + \|g-g_{\e}\|_{\B} \beta(z,\zeta),  \qquad \,  z,\zeta \in \D.
 \]
A natural question is whether a Bloch function $g$ satisfying (\ref{AMaincond}) must belong to the closure of $H^{\infty}$ in $\B$. It was recently conjectured in \cite{AdBa} that  if $e^{\Re(\la g)}$ is a $B_{2}$-weight, for all $\la \in \C$, then $g$ belongs to the closure of $H^{\infty}$ in $\B$. Unfortunately, this turns out to be false and is essentially the main content of our next result. The problem of characterizing the closure of $H^{\infty}$ in $\B$ was initially posed in 1974 by Anderson, Clunie and Pommerenke (see \cite{AnClPo}) and remains open to this date. It is worth mentioning that several variants of this problem have been studied in \cite{NicGal}, \cite{GaMoPa}, \cite{GhaZh}, \cite{NicGib} and \cite{GibSak}. For $0<p<\infty$, let $H^p$ denote the classical Hardy of analytic functions $f$ in $\D$, satisfying 
\begin{equation*}
    \sup_{0<r<1} \int_{\partial \D} |f(r \xi)|^p dm(\xi) < \infty. 
\end{equation*}

\begin{thm}\label{MThm3} There exists a function $g\in \B$ such that $g$ satisfies condition \eqref{AMaincond}, but $g$ does not belong to the closure of $H^p \cap \B$ in $\B$, for any $0< p \leq \infty$.
\end{thm}
\noindent 
It is worth mentioning that \thmref{T1} and Theorem \ref{MThm3} indicate that harmonicity is not preserved in Corollary \ref{closure}. That is, there exists a $B_2$-weight $w$ with the property that $\log w$ is harmonic in $\D$, but the $L^{\infty}(\D)$-approximates in (i) or (ii) of \thmref{T1} are not harmonic. In other words, there exists $g \in \B$ which belongs to the closure of $L^\infty (\D)$ in $BMO(\D)$ and in $\textup{HLip}(\D)$, but it is not in the closure of $H^\infty$ in any of these (semi-)norms. \\ 

\noindent
The paper is organized as follows. Several auxiliary results which may be of independent interest, are collected in Section 2. Section 3 contains the proofs of  \thmref{T1}, Corollary \ref{closure} and \thmref{MThm3}. Finally, in section 4, we briefly mention some applications to the spectra of Cesar\'o operators on the Bergman space. The letter $C$ will denote an absolute positive constant whose value may change from line to line and $C(w)$ will denote a constant depending on $w$.  

 
 \section{Preliminary Results.
 }
 
 We start with an auxiliary Lemma whose proof is a tailor-made version of a more general result which can be found in \cite{AleCon} (See Lemma 2.1).
%
%
\begin{lemma} \label{B1*}Let $w$ be a $B_{2}$-weight on $\D$ of bounded hyperbolic oscillation.  
Then there exists a constant $C(w)>0$ such that
\begin{equation*}\label{B1*est} \int_{\D} \left(w\circ\phi_{z}\right)(\zeta) dA(\zeta) \leq C(w) w(z), \qquad \,  z\in\D.
\end{equation*} 
\end{lemma}
%
%
\begin{proof} 
Notice that the change of variable $\zeta \mapsto \phi_{z}(\zeta)$ gives
\begin{equation}\label{wcompz} \int_{\D} \left( w\circ \phi_{z}\right)(\zeta) dA(\zeta) = \int_{\D} \frac{(1-|z|^{2})^{2}}{|1-\conj{z}\zeta|^{4}} w(\zeta) dA(\zeta).
\end{equation}
Let $D(z)$ denote the disc $D(z)= \left\{ \zeta\in \D: |\zeta-z| < \frac{1}{2}(1-|z|^{2})\right\}$ and observe that by the mean-value theorem for harmonic functions, we can write
\[ \frac{1}{(1- \conj{z}\zeta)^{2}} = \fr{1}{A(D(z))} \int_{D(z)}\fr{1}{(1-\conj{\eta}\z)^{2}} dA(\eta) = P(\chi(z,\cdot))(\zeta), \quad z, \zeta \in \D,
\]
where $\chi(z,\eta) = \frac{1}{A(D(z))} 1_{D(z)}(\eta)$ and $P$ denotes the Bergman projection. Recall that $w\in B_{2}$ if and only if $P:L^{2}(\D,w dA) \rightarrow L^{2}(\D, w dA)$ is bounded, hence there exists a constant $C(w)>0$ only depending on $w$, such that  
\[   \int_{\D} |P(\chi(z,\cdot)|^{2} w dA \leq C(w)  \int_{\D} |\chi(z, \cdot) |^{2} w dA, \qquad  \,  z\in \D.
\]
Going back to  (\ref{wcompz}) with these observations in mind, we obtain
\begin{equation}\label{b1*}\int_{\D} \left( w\circ \phi_{z}\right)(\zeta) dA(\zeta) \leq C(w) \frac{1}{A(D(z))} \int_{D (z)}w(\zeta) dA(\zeta).
\end{equation}
Using the fact that $w$ is of bounded hyperbolic oscillation, the rightmost integral in (\ref{b1*}) is comparable to $w(z)$, which completes the proof.
\end{proof}
%
%
\noindent Our next result establishes the conformal invariance of Bekoll\'e-Bonami weights.
%
%
\begin{lemma}\label{BWconinv}The collection of $B_{2}$-weights is conformally invariant. In fact, if $w$ is a $B_{2}$-weight, there exists a constant $C(w)>0$, such that $ \left[ w\circ \phi_{z} \right]_{B_{2}}   \leq C(w)$ for any $z \in \mathbb{D}$. 
\end{lemma}
\begin{proof}According to the well-known result of Bekoll\'e-Bonami in \cite{BeBo}, it suffices to prove that there exists a constant $C(w)>0$, independent of $z \in \D$, such that 
\begin{equation}\label{b2}
    \int_{\D} | P(f)(\zeta)|^{2} (w\circ \phi_{z})(\zeta) dA(\zeta) \leq C(w) \int_{\D} |f(\zeta)|^{2} (w\circ \phi_{z})(\zeta) dA(\zeta),
\end{equation}
for all  $f \in L^{2}\left(\D, (w\circ \phi_{z})dA\right)$, $z\in \D$. Notice that the change of variable $\zeta \mapsto \phi_{z}(\zeta)$ gives 
\begin{equation}\label{B2coninv1} \int_{\D} | P(f)|^{2} (w\circ \phi_{z}) dA= \int_{\D} |\left(P(f) \circ \phi_{z}\right)  \cdot \phi_{z}'|^{2} \, w dA.
\end{equation}
We now claim that for any $z,\z \in \D$, the following identity holds:
\begin{equation}\label{B2confid} \left(P(f) \circ \phi_{z}\right) (\zeta) \cdot \phi_{z}'(\zeta) = P \left( \left(f\circ \phi_{z}\right) \cdot \phi_{z}' \right) (\zeta).
\end{equation}
To prove this, we primarily observe that a straightforward calculation shows 
\[ \frac{\phi_{z}'(\z)}{(1-\overline{\eta}\phi_{z}(\z))^{2}} = \frac{\overline{\phi_{z}'(\eta)}}{(1-\overline{\phi_{z}(\eta)}\zeta)^{2}}.
\]
Now using this, we can write 
\begin{equation*} 
\begin{split} \left(P(f) \circ \phi_{z}\right) (\zeta) \cdot \phi_{z}'(\zeta) = \int_{\D} \frac{f(\eta) \cdot \phi'_{z}(\zeta)}{(1-\conj{\eta}\phi_{z}(\zeta))^{2}}dA(\eta) 
=  \int_{\D} \frac{f(\eta) \cdot \conj{\phi'_{z}(\eta)}}{(1-\conj{\phi_{z}(\eta)}\zeta)^{2}}dA(\eta).
\end{split}
\end{equation*}
Performing the change of variable $\eta \mapsto \phi_{z}(\eta)$ in the last integral, we can simplify it as  
\[ \int_{\D} \frac{\left( f\circ \phi_{z}\right)(\eta) \cdot \conj{\phi_{z}'(\phi_{z}(\eta))} \cdot |\phi_{z}'(\eta)|^{2} }{(1-\conj{\eta}\zeta)^{2}} dA(\eta) = 
P \left( \left(f\circ \phi_{z}\right) \cdot \phi_{z}' \right) (\zeta).
\]
This establishes the identity in (\ref{B2confid}). Now using this together with the result of Bekoll\'e-Bonami, we can find a constant $C(w)>0$, such that 
\[ \int_{\D} |P \left( \left(f\circ \phi_{z}\right) \cdot \phi_{z}' \right) (\zeta)|^{2} w(\zeta) dA(\zeta) \leq C(w) \int_{\D}  |\left( f\circ \phi_{z}\right) (\zeta) \cdot \phi_{z}'(\zeta) |^{2} w(\zeta) dA(\zeta) 
\] 
\[ = C(w) \int_{\D} |f(\zeta)|^{2} \left( w\circ \phi_{z}\right)(\zeta) dA(\zeta).
\]
In the last step, we again utilized the change of variable $\zeta \mapsto \phi_{z}(\zeta)$. According to \eqref{B2coninv1} and \eqref{B2confid}, we finally obtain
\begin{equation}\label{ab2}\int_{\D}|P(f)|^{2}\left( w\circ \phi_{z} \right) dA \leq C(w) \int_{\D} |f|^{2} \left(w\circ \phi_{z}\right)dA,
\end{equation}
for all $f\in L^{2}\left(\D, (w\circ\phi_{z})dA\right)$, 
which proves the desired claim in \eqref{b2}.
\end{proof}
%
%
\noindent Next, we state an elementary result for future reference. Given a Carleson square $Q=Q(I) \subset \D $, we denote by $z_Q$ the point $z_Q = (1 - m(I)) \xi (I)$, where $\xi (I)$ is the center of the arc $I \subset \partial \D$.  
%
%
\begin{lemma}\label{LipB2}
Let $v:\D \rightarrow \R$ be a hyperbolic Lipschitz function with $\|v \|_{\text {HLip}} < 1$. Then $e^{2v}  \in B_2$. 
\end{lemma}
\begin{proof}
Note that there exists a constant $M>0$ such that for any Carleson square $Q$ and any $z \in Q$, we have  
\begin{equation*}
    \left|\beta (z, z_Q) - \frac{1}{2} \log \frac{1-|z_Q|}{1-|z|} \right| \leq M . 
\end{equation*}
Hence for any Carleson square $Q$ and any $z \in Q$, we can write 
\begin{equation}\label{equacio}
e^{-M} \left(\fr{1-|z|}{1-|z_Q|} \right)^{\|v \|_{\text {HLip}}/2} \leq e^{v(z) - v(z_Q)} \leq e^M  \left(\fr{1- |z_Q|}{1-|z|} \right)^{ \|v \|_{\text {HLip}} /2}.  
\end{equation}
Since $\|v \|_{\text {HLip}} < 1$, there exists a constant $C= C(\|v \|_{\text {HLip}} ) >0$, such that
\[ \int_Q \left(\fr{1- |z_Q|}{1-|z|} \right)^{ \|v \|_{\text {HLip}} } \leq C  A(Q) . 
\]
We deduce that 
\[ \frac{1}{A(Q)} \int_Q e^{2v} dA \leq C e^{2M} e^{ 2 v(z_Q)}. 
\]
Using the left inequality in \eqref{equacio} we also obtain
\[ \frac{1}{A(Q)} \int_Q e^{-2v} dA \leq C e^{2M} e^{- 2v(z_Q)}. 
\]
This finishes the proof. 
\end{proof}
%
%
\noindent 
Functions in $BMO(\D)$ satisfy the John-Nirenberg inequality. That is, there exist universal constants $C_1 >0$ and $C_2 >0$ such that for any $f \in BMO (\D)$ and any disc $D$ centered at a point in $\D$, we have
\begin{equation}\label{J-N}
    \frac{A \left( \left\{ z \in D \cap \D : |f(z) - f_{D \cap \D}| > \lambda \right\} \right)}{A (D \cap \D)} \leq C_1 e^{-C_2 \lambda / \|f\|_{BMO(\D) } }, \quad \lambda >0 . 
\end{equation} 
Let us recall the well known relation between $BMO(\partial \D)$ and Muckenhoupt weights on $\partial \D$. If $w$ is a Muckenhoupt $A_2$-weight, then $\log w \in BMO (\partial \D)$. Conversely there exists a universal constant $c>0$ such that if $f \in BMO(\partial \D)$ with $\|f\|_{BMO(\partial \D)} = 1$,  then $e^{cf}$ is a $A_2$-weight. Our next result establishes an analogue for Bekoll\'e-Bonami weights. 
%
%
\begin{lemma}\label{BB-BMO}
\begin{itemize}
\item[(i)]
There exist absolute constants $C_1>0$, $C_2 >0$, such that 
\begin{equation*}
    [ e^{C_2 f / \|f\|_{BMO (\D)}} ]_{B_2} \leq C_1 , 
\end{equation*}
for any $f \in BMO (\D)$. 
\item[(ii)] For any $f \in VMO (\D)$, we have 
\begin{equation*} \lim_{\delta \rightarrow 0+} \sup_{A(Q)< \delta} \left(\frac{1}{A(Q)}\int_{Q}e^{f}dA \right) \left(\frac{1}{A(Q)} \int_{Q} e^{-f} dA \right) =1,
\end{equation*}
where the supremum is taken over all Carleson squares $Q \subset \D$.
\item[(iii)] For any $w\in B_{2}$ and any $\la >  2 + \log \left[ w\right]_{B_{2}} $,  one has 
\begin{equation}\label{expdec}
   A(\{z \in Q : |\log w(z) - (\log w)_Q | > \la \} ) \leq 2 e^2 \left[w \right]_{B_{2}}^2 e^{-\la}A(Q),
\end{equation}
for any Carleson square $Q \subset \D$. Moreover, if $w \in B_2$ is of bounded hyperbolic oscillation, then $\log w \in BMO (\D)$.
\end{itemize}
\end{lemma}
%
%
\noindent 
It is worth mentioning that in contrast to the John-Nirenberg inequality \eqref{J-N}, no constant depending on $w$ appears in the exponent of \eqref{expdec}. Roughly speaking, this indicates that the logarithm of a $B_2$ weight behaves as a function of fixed $BMO$-norm on Carleson squares. 
%
%
\begin{proof}
(i) We can assume $\|f\|_{BMO (\D)} = 1$. Observe that for any Carleson square $Q$ there exists a disc $D$ centered at a point in $\D$ such that $Q \subset D$ and $A(D) \leq 4 A(Q)$. Note that 
\begin{equation*}
    |f_{D \cap \D} - f_Q| \leq \frac{1}{A(Q)} \int_Q |f - f_{D \cap \D} | dA \leq  \frac{4}{A(D)} \int_{D \cap \D} |f - f_{D \cap \D} | dA \leq 4.
\end{equation*}
Then the John-Nirenberg inequality \eqref{J-N}  provides two absolute constants $C_1 >0$, $C_2 >0$ such that 
\begin{equation*}
    \frac{1}{A(Q)} \int_Q e^{C_2 |f - f_Q|} dA \leq C_1, 
\end{equation*}
for any Carleson square $Q \subset \D$. Then 
\begin{equation*}
    \frac{1}{A(Q)} \int_Q e^{C_2 (f - f_Q)} dA \leq C_1 \qquad \text{and} \qquad \frac{1}{A(Q)} \int_Q e^{-C_2 (f - f_Q)} dA \leq C_1 . 
\end{equation*}
Hence
\begin{equation*}
     \left( \frac{1}{A(Q)} \int_Q e^{C_2 f} dA  \right) \left(\frac{1}{A(Q)} \int_Q e^{- C_2 f} dA \right) \leq C_1^2. 
\end{equation*}
(ii) Now let $f\in VMO(\D)$ and notice that by the preceding argument in (i), it is enough to show that 
\begin{equation}\label{RedSB2} \lim_{\delta \rightarrow 0+} \sup_{A(Q) < \delta} \frac{1}{A(Q)}\int_{Q}e^{|f(z)-f_{Q}|} dA(z) =1.
\end{equation}
To this end, let $p$ be a polynomial and $q=f-p$. Observe that for any Carleson square $Q$ and $z \in Q$, we have
\[ |f(z)-f_{Q}| \leq |q(z)- q_{Q}| + \omega_{p}(Q).
\]
Here $\omega_{p}(Q):=\sup_{z,\zeta\in Q}|p(z)-p(\zeta)|$ denotes the oscillation of $p$ on $Q$. This gives 
\[ \frac{1}{A(Q)}\int_{Q}e^{|f(z)-f_{Q}|} dA(z) \leq e^{\omega_{p}(Q)} \frac{1}{A(Q)}\int_{Q} e^{|q - q_{Q}|}dA.
\]
We now apply the John-Nirenberg inequality in \eqref{J-N} to $q$, to find absolute constants $c_{1}, c_{2}>0$, such that 
\[ \frac{A \left( \left\{ z \in Q: |q(z) - q_{Q}| > \lambda \right\}\right)}{A (Q)} \leq c_1 e^{-c_2 \lambda / \|q \|_{BMO(\D) } }, \quad \lambda >0 . 
\]
for every Carleson square $Q$. With this at hand and by choosing $p$, such that $\|q \|_{BMO(\D)}< c_{2}/2$, we get
\[ \frac{1}{A(Q)}\int_{Q} e^{|q - q_{Q}|}dA \leq 
1 + \frac{c_{1}\|q \|_{BMO(\D)}}{c_{2} - \|q \|_{BMO(\D)}} \leq 1 + \fr{2c_{1}}{c_{2}} \|q \|_{BMO(\D)}
\].
Now combining, we arrive at 
\[ \frac{1}{A(Q)}\int_{Q}e^{|f(z)-f_{Q}|} dA(z) \leq e^{\omega_{p}(Q)} \left(1 + \frac{2c_{1}}{c_{2}}\|q\|_{BMO(\D)}\right).
\]
By uniform continuity of polynomials, it follows $\lim_{\delta \rightarrow 0+} \sup_{A(Q)<\delta} \omega_{p}(Q) =0$. Finally, choosing $p$ to be an approximate of $f$ in $BMO(\D)$ establishes \eqref{RedSB2}, thus (ii) follows. \\ \\
(iii) We first prove the estimate in \eqref{expdec}. Let $Q$ be a Carleson square.  By Chebyshev's inequality, we have for any $t >1$, 
\begin{equation*}
    A \left(\{z \in Q : w(z) > t w_Q \}\right) \leq \frac{A(Q)}{t}. 
\end{equation*}
Since $w^{-1}_Q \leq \left[w\right]_{B_2} / w_Q$, Chebyshev's inequality also gives
\begin{equation*}
    A \left( \{z \in Q : w(z) <  \frac{ w_Q}{t \left[w\right]_{B_2}} \}\right) \leq \frac{A(Q)}{t} , \quad t>1. 
\end{equation*}
Since $\left[w\right]_{B_{2}}\geq 1$, we obtain for any $t>1$
\begin{equation*}
    A \left(\{z \in Q : |\log w(z) - \log (w_Q )| > \log (t \left[w\right]_{B_2}) \}\right) \leq  \frac{2A(Q)}{t} . 
\end{equation*}
Set $\la = \log (t \left[w\right]_{B_2})$ to deduce that for any $\la > \log  \left[w\right]_{B_2} $, we have 
\begin{equation}\label{jn}
    A \left(\{z \in Q : |\log w(z) - \log (w_Q )| > \la \}\right) \leq 2 \left[w \right]_{B_2} e^{- \la} A(Q). 
\end{equation}
Observe that \eqref{jn} gives 
\begin{equation*}
\begin{split}  |(\log w )_{Q} -\log( w_{Q}) | \leq \frac{1}{A(Q)}\int_{Q}| \log w - \log (w_{Q}) | dA \\ =\frac{1}{A(Q)} \int_{0}^{\infty} A \left(\{z \in Q : |\log w(z) - \log (w_Q )| > \la \}\right) d\la \\ = \int_{0}^{\log \left[ w\right]_{B_{2}} } d\la + 2 \left[ w \right]_{B_{2}} \int_{\log \left[ w\right]_{B_{2}} }^{\infty} e^{- \la } d \la =  \log \left[ w\right]_{B_{2}} + 2 : = C(w).
\end{split}
\end{equation*}
With this at hand, we see that whenever $\la > C(w)$, the following set inclusion holds
\[ \{z \in Q : |\log w(z) - (\log w)_Q | > \la \} \subset \{z \in Q : |\log w(z) - \log (w_Q) | > \la - C(w) \}.
\]
Applying \eqref{jn}, we obtain 
\[ A(\{z \in Q : |\log w(z) - (\log w)_Q | > \la \} ) \leq 2 e^2 \left[w \right]_{B_{2}}^2 e^{-\la}A(Q),
\]
for all $\la > C(w)$, hence proving $\eqref{expdec}$. Now under the additional assumption that $w$ is of bounded hyperbolic oscillation we will prove that $\log w \in BMO (\D)$. Let $D$ be a disc of radius $r>0$ centered at a point $z \in \D$. If $r \leq (1-|z|)/2$, we use the fact that $w$ has bounded hyperbolic oscillation, to find a constant $C(w)>0$, independent of $D$ such that $|\log w - \log w(z)| \leq C(w)$ on $D$. Since 
\begin{equation*}
   |\log w(z) - (\log w)_D | \leq  \frac{1}{A(D)} \int_D |\log w - \log w (z)| dA \leq C(w), 
\end{equation*}
we deduce 
\begin{equation*}
    \frac{1}{A(D)} \int_D |\log w - (\log w)_D | dA \leq 2C(w). 
\end{equation*}
\noindent 
If $r > (1-|z|)/ 2$, consider a Carleson square $Q$ with $D \cap \D \subset Q$ and $A(Q) < 4 A (D)$. Then 
\begin{equation*}
     \frac{1}{A(D)} \int_{D \cap \D}  |\log w - (\log w)_Q| dA \leq \frac{4}{A(Q)} \int_{Q}  |\log w - (\log w)_Q| dA , 
\end{equation*}
which, by \eqref{expdec}, is uniformly bounded. Hence $\log w \in BMO (\D)$.
\end{proof}
%
%
\noindent Recall that a locally integrable function $f$ defined on $\R^d$ belongs to $BMO (\R^d)$, if 
\begin{equation*}
    \|f\|_{BMO(\R^d)} = \sup \frac{1}{m_d (R)} \int_{R} |f(x) - f_{R}| dm_d (x) < \infty , 
\end{equation*}
where the supremum is taken over all cubes $R \subset \R^d$ and $m_d$ is Lebesgue measure on $\R^d$. A crucial observation is that if $f \in BMO (\D)$, then the extension $f^*$ defined by 
\begin{equation}\label{BMOext}
f^*(z)=\begin{cases} 
f(z) & ,z \in \D \\ 	
f(1/\overline{z}) & ,|z|>1,
\end{cases}
\end{equation}
belongs to $BMO({\R}^2)$. Given $f\in BMO({\R}^d)$, denote by $\varepsilon (f)$ the infimum of $\varepsilon >0$, for which there exists a constant $\lambda (\varepsilon) >0$, such that 
\begin{equation*}
   m_d \left( \left\{ x \in Q : |f(x) - f_Q| > \lambda \right\} \right) \leq e^{ - \lambda / \varepsilon } m_d (Q), 
\end{equation*}
for any $\lambda > \lambda (\varepsilon)$ and any cube $Q \subset {\R}^d$. 
Garnett and Jones proved in \cite{GaJo} that there exists a constant $C>0$, only depending on the dimension such that for any $f \in BMO (\R^d)$, one has
\begin{equation}\label{GJ}
    C^{-1} \varepsilon (f) \leq \inf \{\|f-h  \|_{BMO (\R^d)} : h \in L^{\infty} (\R^d)  \} \leq  C \varepsilon (f) . 
\end{equation} 
This result will be used in the proof of part (ii) of Theorem \ref{T1}. \\ \\
\noindent In contrast to the theory of Muckenhoupt weights, where $w\in A_{2}$ implies that $\log w \in BMO(\partial \D)$, one cannot expect the same to hold for general $B_{2}$-weights, since they can posses very wild behavior inside $\D$. However, under some rigid assumptions one can in fact retrieve  similar results. In fact, our next result provides a relationship between Bekoll\'e-Bonami weights which are exponentials of harmonic functions and Bloch functions. Our next Proposition is reminiscent of the deep interplay between $A_{2}$-weights on $\partial \D$ and $BMO(\partial \D)$ as previously mentioned. Meanwhile, part (ii) gives a new characterization of $\B_{0}$, in a similar spirit to the characterization of vanishing mean oscillation by Sarason in \cite{Sar}.
%
%
\begin{prop}\label{HBekBlo} Let $f:\D \rightarrow \R$ be analytic. 
\begin{itemize}
\item[(i)] Then, $f\in \B$  if and only if $e^{\delta \Re(f)}$ is a $B_{2}$-weight, for some $\delta >0$.
\item[(ii)] Then,  $f\in \B_{0}$ if and only if
\begin{equation}\label{V2} \lim_{\delta \rightarrow 0+} \sup_{A(Q) < \delta} \left(\frac{1}{A(Q)}\int_{Q}e^{\Re(f)}dA \right) \left( \frac{1}{A(Q)} \int_{Q} e^{-\Re(f)}dA \right) = 1.
\end{equation} 
\end{itemize}
\end{prop}
\noindent Notice that analyticity seems to compensate for the a priori assumption of the weights being of bounded hyperbolic oscillation. In order to prove Proposition \ref{HBekBlo}, we will need to establish two lemmas.
%
%
\begin{lemma}\label{Cauchyosc}Let $f:\D\rightarrow \C$ be an analytic function  and $u=\Re(f)$. Then for any $3/4 <|z|< 1$, there exists a Carleson square $Q_{z}\subset \D$ with $A(Q_{z}) \leq 4 (1-|z|)^2$, such that
\begin{equation*} (1-|z|^{2})^{2}|f'(z)|^{2} \leq \frac{C}{A(Q_{z})} \int_{Q_{z}} |u-u_{Q_{z}}|^{2}dA,
\end{equation*}
for some universal constant $C>0$.
\end{lemma}
\begin{proof} Notice that by Cauchy's integral formula, we can write
\[ f'(z) = \frac{1}{2\pi } \int_{0}^{2\pi} \frac{f(z+ re^{it}) - \alpha}{re^{it}} dt ,  \qquad z\in \D , 
\]
for $0 < r < 1-|z|$ and all $\alpha \in \C$. Applying Cauchy-Schwartz inequality, we get
\[ |f'(z)|^{2} \leq \frac{1}{2\pi r^2} \int_{0}^{2\pi} |f(z+re^{it})-\alpha|^{2} dt. 
\]
Since the conjugate operator is an isometry on the Hardy space $H^{2}$, we can find an absolute constant $C>0$, such that
\[
|f'(z)|^{2}\leq  \frac{C}{r^2} \int_{0}^{2\pi} |u(z+re^{it})-\Re(\alpha)|^{2} dt.
\]
Now integrating this inequality with respect to $r \in [(1-|z|)/2 ,  1-|z|]$, we get
\[
(1-|z|^{2})^{2}|f'(z)|^{2} \leq \frac{C}{(1-|z|)^{2}} \int_{ \{ (1-|z|)/2 \leq |z-\z| < 1-|z| \}} |u(\z)-\Re(\alpha)|^{2}dA(\z).
\]
Since $1-|z|< 1/4$, we can take $Q_z$ to be a Carleson square containing the disc $\{\z : |z-\z| < 1-|z| \}$ with $A(Q_{z}) \leq 4(1-|z|)^2$. Moreover, choosing $\alpha = f_{Q_{z}}$, we obtain
\[ (1-|z|^{2})^{2}|f'(z)|^{2} \leq \frac{4 C}{A(Q_{z})}\int_{Q_{z}} |u(\z)-u_{Q_{z}} |^{2} dA(\z). 
\]
\end{proof}
%
%
\noindent Our next lemma is a refined version of an abstract measure theoretic lemma, attributed to D. Sarason (See \cite{Sar}, Lemma 3).
%
%
\begin{lemma} \label{localSB2} Let $(\Omega,\mu)$ be a probability space and let $w$ be a positive integrable function on $\Omega$ such that $w^{-1} $ is also  integrable. Assume 
\begin{equation}\label{small char}  \left( \int_{\Omega} w d\mu \right) \left( \int_{\Omega} w^{-1} d\mu \right) =1 + \e,
\end{equation}
for some $0<\e<1$
Then 
\begin{equation}\label{Small MO}  \int_{\Omega} | \log w - \int_{\Omega} \log w \,d\mu \,  |^{2} \, d\mu \leq 4 \, \e.
\end{equation}
\end{lemma}
%
%
\begin{proof}
By means of multiplying $w$ with a positive scalar, thus not affecting the quantity in \eqref{small char}, we may assume that $\int_{\Omega} w d\mu=1$. Consider the set
\[ S_{\e} = \left\{\omega \in \Omega: \fr{1}{1+\e} \leq w(\omega) \leq 1+\e \right\},
\]
and notice that for $\omega \in S_{\e}$, we have that $|\log w(\omega) | \leq \log(1+\e) \leq \e$. Moreover, from the elementary inequality $2+t^{2} \leq e^{t}+e^{-t}$, $t \in \mathbb{R}$, we get that for every positive function $w$, the estimate $ \log^{2} w \leq w + w^{-1}-2$ holds. With these observations at hand, we can write 
\begin{equation}\label{MOest1}  \int_{\Omega}  \log^{2} w  d\mu \leq \e^{2} + \int_{\Omega \setminus S_{\e}}\left( w + w^{-1} \right) d\mu - 2\mu(\Omega \setminus S_{\e}).
\end{equation}
According to \eqref{small char} and the assumption $\int_{\Omega} w d\mu=1$, we have
\begin{equation*}\label{Cest} \int_{\Omega \setminus S_{\e}}\left( w + w^{-1} \right) d\mu = 2+\e - \int_{S_{\e}}( w+w^{-1}) d\mu \leq 2+ \e - \fr{2}{1+\e} \mu(S_{\e}).
\end{equation*}
Hence going back to \eqref{MOest1}, we obtain 
\begin{equation*}\label{MOest2}  \int_{\Omega}  \log^{2} w  d\mu \leq \e^2 + \e +2(1-\fr{1}{1+\e}) \mu(S_{\e}) \leq 4 \e.
\end{equation*}
It now immediately follows that
\[ \int_{\Omega} | \log w - \int_{\Omega} \log w d\mu |^{2} d\mu = \int_{\Omega} \log^{2} w d\mu - \left( \int_{\Omega} \log w d\mu \right)^{2} \leq 4\e.
\]
\end{proof}
%
%
%
\begin{proof}[Proof of Proposition \ref{HBekBlo}]
Note that there exists a constant $C>0$ such that $\|f\|_{BMO(\D)} \leq C \|f\|_{\B} $. Similarly $f \in VMO (\D)$ if $f \in {\B}_0$. Hence the left to right implications of (i) and (ii) are immediate consequences of Lemma \ref{BB-BMO} and the intrinsic content of Proposition \ref{HBekBlo} is the "if" part of the statements. 
\noindent
(i) In order to prove the converse implication, we assume $e^{u}$ is a $B_{2}$-weight with $u:\D \rightarrow \R$ harmonic. Notice that for any Carleson square $Q \subset \D$, we have
\[ \frac{1}{A(Q)} \int_{Q} e^{|u-u_{Q}|}dA \leq \frac{1}{A(Q)} \int_{Q} e^{u-u_{Q}}dA + \frac{1}{A(Q)} \int_{Q} e^{u_{Q}-u}dA.
\] 
From Jensen's inequality and the $B_{2}$-condition on $e^{u}$, it follows that 
\[ e^{u_{Q}} \leq \frac{1}{A(Q)} \int_{Q}e^{u}dA \leq \left[e^{u}\right]_{B_{2}} e^{u_{Q}}.
\]
We deduce that there exists $C(u)>0$, such that 
\[ \sup_{Q}\frac{1}{A(Q)} \int_{Q} e^{|u-u_{Q}|}dA \leq C(u).
\]
Let $\Re(f)=u$ and observe that by Lemma \ref{Cauchyosc} and the simple inequality $t^{2} \leq 2e^{t}$, for $t>0$, there exists a constant $C_1 (u)>0$ such that 
\[ \sup_{3/4 < |z| < 1} (1-|z|^{2})^{2} |f'(z)|^{2} \leq C_1 (u).
\]
This is enough to conclude that $f\in \B$.
\\
(ii) Again, let $u=\Re(f)$ and notice that by Lemma \ref{Cauchyosc}, it follows that for sufficiently small $\delta>0$, we have
\[ \sup_{1-|z| \leq \delta} (1-|z|^{2})^{2}|f'(z)|^{2} \leq \sup_{A(Q)< 4 \delta^2}\frac{C}{A(Q)}\int_{Q} |u(\z)-u_{Q} |^{2} dA(\z),
\]
Now assuming that condition \eqref{V2} holds, we have according to Lemma \ref{localSB2}, that
\begin{equation*}\label{Voscu} \lim_{\delta \rightarrow 0+} \sup_{A(Q)< \delta} \frac{1}{A(Q)} \int_{Q} |u(z)-u_{Q}|^{2}dA(z) = 0.
\end{equation*} 
This proves that $f\in \B_{0}$.
\end{proof}
%
%

%
%

\section{Proofs of Main Results.}
\noindent We are now ready to prove Theorem \ref{T1}. 

\begin{proof}[Proof of Theorem \ref{T1}] 
(i) We first show that 
\begin{equation}\label{4}
    \inf \{\|f-h\|_{\text {HLip}} : h \in L^\infty (\D) \} \leq 4 \gamma (f)  . 
\end{equation}
Given $\z \in \D$, let $D_{\z}$ be the disc centered at $\z$ of radius $(1-|\z|^2) / 4$. Since the weight $w = e^f$ is of bounded hyperbolic oscillation, so is $(w\circ \phi_{z})^{\la}$, for all $z\in \D$, $\la>0$, by conformal invariance of \eqref{WCHD}. Consequently, there exits a constant $C_1 (\lambda) >0$, such that 
\begin{align*}
    \left(w\circ \phi_{z}\right)^{\la}(\z) &  \leq C_1 (\lambda) \fr{1}{A(D_{\z})}\int_{D_{\z}}  \left(w\circ \phi_{z}\right)^{\la} dA  \\  & \leq 16 C_1 (\lambda) \left(\fr{1+|\z|^{2}}{1-|\z|^{2}}\right)^{2} \int_{\D}  \left(w\circ \phi_{z}\right)^{\la} dA, 
\end{align*}
for all $z,\z \in \D$. Now assuming that $w^{\la}$ is a $B_{2}$-weight, we may apply Lemma \ref{B1*} to $w^{\la}$. This provides a constant $C(\la)>0$, such that 
\[ \left(w\circ \phi_{z}\right)^{\la}(\z) \leq C(\la) \left(\fr{1+|\z|^{2}}{1-|\z|^{2}} \right)^{2} w(z)^{\la}, \qquad z,\z \in \D.
\]
Now performing the change of variable $\z \mapsto \phi_{z}(\z)$ and taking logarithms, we conclude that
\begin{equation}\label{petit2}
    |\log w (\z) - \log w (z)| \leq \frac{\log C(\lambda)}{\lambda} +  \frac{4 \beta(\z , z)}{\lambda}, \quad  \z , z \in \D . 
\end{equation}
Set $\e = 1/ \lambda$ and $C_{\e} = \log C(\lambda) / \lambda$. Fix $\eta >0$ and let  $\Lambda=\left\{z_{j}\right\}^{\infty}_{j= 1}$ be a sequence of points in $\D$, with the properties that $\beta(z_{k},z_{j}) \geq C_{\e}/\eta$, for all $k \neq j$, and such that $\inf_{\la \in \Lambda} \beta(z, \la) \leq 10 C_{\e}/\eta$, for any $z\in \D$. See \cite{DurSch} for details on such constructions. By \eqref{petit2}, we have that \[ \abs{ \log w(z_{k}) -\log w (z_{j}) } \leq C_{\e} + 4 \e \beta(z_{k},z_{j}) \leq (4 \e + \eta) \beta(z_{k},z_{j}),\quad   j\neq k.
\] 
This means that the map $\log w : \Lambda \rightarrow \mathbb{R}$ defined by $z_{j}\mapsto \log w (z_{j})$ is Lipschitz continuous with respect to the hyperbolic metric $\beta$ on $\Lambda \subset \D$ and the euclidean metric in $\mathbb{R}$, with norm less than $(4 \e + \eta)$. We now use the McShane-Valentine extension Theorem (see \cite{Mc}) to find a function $g= g_{\e, \eta}:\D \rightarrow \R$, with $\|g\|_{\text {HLip}} \leq 4 \e + \eta$, such that $g (z_j ) = \log w (z_j)$ for $j=1,2,\ldots$. We claim that  $h:= \log w - g$ belongs to $L^{\infty}(\D)$. To this end, fix an arbitrary $z\in \D$ and observe that by construction of $\Lambda$, we can always pick $z_{j}\in \Lambda$, with $\beta(z, z_{j}) \leq 10C_{\e} /\eta$. Then, using \eqref{petit2}, we obtain $|h(z)| = |h(z) - h(z_j)| \leq C_\e + (8 \e + \eta)  \beta (z, z_j) \leq C_\e + 10 (8 \e + \eta)  C_{\e} / \eta $ and we deduce $h \in L^{\infty} (\D)$. Since  $\|\log w - h \|_{\text {HLip}} \leq 4 \e + \eta$ and $\eta >0$ can be taken arbitrary small, this proves \eqref{4}.


\noindent Conversely, we now prove the first estimate in part (i) of Theorem \ref{T1}. By Lemma \ref{LipB2}, for any $0< t < 1 $ and any $h \in L^\infty (\D)$ such that $f-h \in  \textup{HLip}(\D)$, we have
\begin{equation*}
    e^{2 t (f-h) / \|f - h \|_{\text {HLip}}} \in B_2 . 
\end{equation*}
Hence $\gamma (f) \leq \|f - h\|_{\text {HLip }} / 2t$. We deduce that 
\begin{equation*}
    \gamma (f) \leq \frac{1}{2} \inf \{\|f-h\|_{\text {HLip}} : h \in L^\infty (\D) \}. 
\end{equation*}
\noindent 
We now turn our attention to proving (ii). Note that $\gamma (f) = \gamma (f-h)$ for any $h \in L^\infty (\D)$. Part (i) of Lemma \ref{BB-BMO} gives that $\gamma (f) = \gamma (f-h) \leq C_2^{-1} \|f-h \|_{BMO (\D)}
$, for any $h \in L^\infty (\D)$. Hence 
\begin{equation*}
 \gamma (f) \leq C_2^{-1} \inf \{\|f-h \|_{BMO(\D)} : h \in  L^\infty (\D) \}.    
\end{equation*}
Note that in this step, we did not need the assumption that $e^{f}$ has bounded hyperbolic oscillation, thus this part of the proof holds for any $f\in BMO(\D)$.
We now establish the converse estimate. Fix  $\alpha >0$ and assume $w_\alpha = e^{f/ \alpha} \in B_2$. Part (iii) of Lemma \ref{BB-BMO} gives that
\begin{equation*}
    A \left( \{z \in Q : \left|\frac{f(z)}{\alpha} - (\log w_\alpha)_Q \right| > t \}\right) \leq 2 e^2 \left[w_{\alpha} \right]_{B_{2}}^2 e^{-t}A(Q), 
\end{equation*}
for any $t >  2 + \log \left[ w_{\alpha} \right]_{B_{2}}$ and any Carleson square $Q \subset \D$. Writing $t^* = \alpha t$, we deduce 
\begin{equation}\label{ED2}
    A \left( \{z \in Q : |f(z) - f_Q| > t^* \} \right) \leq 2 e^2 \left[w_{\alpha} \right]_{B_{2}}^2 e^{-t^*/ \alpha} A(Q), 
\end{equation}
for any $t^* >  \alpha (2 +  \log \left[ w_{\alpha} \right]_{B_{2}})=:t_{0}^*$ and any Carleson square $Q \subset \D$. 
\noindent 
Since $w=e^f$ has bounded hyperbolic oscillation, there exists a constant $C(w)>0$, such that for any disc $D$ of center $z \in \D$ and radius smaller than $(1-|z|)/2$, we have $|f-f(z)|\leq C(w)$ on $D$. This observation and the estimate in \eqref{ED2} imply the existence of a constant $C_1 = C_1 (\alpha, w)>0$, such that for any disc $D$ centered at a point in $\D$, we have
\begin{equation}\label{ED3}
    A \left( \{z \in D \cap \D : |f(z) - f_{ D \cap \D} |  > 2 t^* \} \right) \leq C_1 e^{-t^* / \alpha } A(D), 
\end{equation}
for any $t^* > \max \{t_0^* , 2C(w)\} =: t_1^*$. Consider the extension $f^*$ of $f$ given by \eqref{BMOext}. The estimate \eqref{ED3} gives that there exists a constant $C_2 >0$, such that for any square $R \subset \R^2$, we have
\begin{equation*}
    A \left( \{z \in R : |f^* (z) - f^*_R| >  2 t^* \} \right) \leq C_2 e^{- t^* / 2\alpha} A(R), 
\end{equation*}
for any $t^* > 2 t_1^* $. Now the result of Garnett and Jones \eqref{GJ} implies the existence of an absolute constant $C>0$, such that  
\begin{equation*}
    \inf \{ \|f^* - h \|_{BMO (\R^2)} : h \in  L^\infty (\R^2) \} \leq  C \alpha. 
\end{equation*}
We deduce 
\begin{equation*}
    \inf \{ \|f - h \|_{BMO (\D)} : h \in  L^\infty (\D) \} \leq  C \alpha,  
\end{equation*} 
which finishes the proof. 
\end{proof}
%
\noindent Regarding the proof of \thmref{T1}, a valuable remark is in order. First, notice that as a byproduct of the proof of part (i), an application of the McShane-Valentine theorem \cite{Mc} gives the following description: \\ A weight $w$ has bounded hyperbolic oscillation if and only if there exists $u\in L^{\infty}(\D)$ and $v \in \textup{HLip}(\D)$, such that 
\begin{equation}\label{B+HLip} \log w = u + v.
\end{equation}
Now consider the linear space $\mathcal{S}= \{f: e^{f} \, \text{has bounded hyperbolic oscillation} \}$ equipped with the norm
\[ \|f\|_{\mathcal{S}} = \inf \{\|u\|_{L^{\infty}(\D)} + \|v\|_{\text{HLip}} : f=u+v \}.
\]
It is straightforward to check that $\mathcal{S}$ is a conformally invariant Banach space, contained in $BMO(\D)$. 
We now turn to the proof of Corollary \ref{closure}. 
%
%
\begin{proof}[Proof of Corollary \ref{closure}] Notice that the equivalence of (i), (ii) and (iii) is an immediate consequence of the distance formulas from Theorem \ref{T1}. Now suppose (i) holds. Then a verbatim implementation of the proof of part (i) of \thmref{T1} yields \eqref{petit2}, i.e, for any $\la >0$, there exists a constant $C(\la)>0$, such that
\[ \abs{ \log w(z) -\log w (\z) } \leq C(\la) + \frac{4}{\la} \beta(z,\z),  \qquad  z,\z \in \D.  
\]
This precisely condition \eqref{Maincond} in (iv). Conversely, assume condition  \eqref{Maincond} in (iv) holds.  Fix an arbitrary arc $I\subseteq \partial \D$ with center $\xi_{I}$ and set $z_{I} = (1-m(I))\z_{I}$. Now for any subarc $J\subset I$, a straightforward calculation shows that there exists an absolute constant $M>0$ such that  
\[ \left| \beta(z_{J},z_{I}) - \fr{1}{2}\log \left( \fr{1-|z_{I}|}{1-|z_{J}|}\right) \right| \leq M. 
\]
With this at hand, we may rephrase condition (\ref{Maincond}) as follows: for any $\e>0$, there exists $c_{\e}>0$, such that for any arc $J\subset I$, we have
\begin{equation}\label{thm1est} e^{-c_{\e}} \left( \fr{m(I)}{m(J)} \right)^{-\e} \leq \fr{w(z_{J})}{w(z_{I})} \leq e^{c_{\e}} \left( \fr{m(I)}{m(J)} \right)^{\e}.
\end{equation}
%
%
\begin{figure}\label{fig1}
\centering
\begin{tikzpicture}
\draw (0,0) rectangle (4,4) node [above] at (2, 4) {$Q_{I}$};
\draw (0,2) -- (4,2) node [below] at (2,3) {$T_{I}$};
\draw (0,1) -- (4,1); 
\draw (0, 1/2) -- (4, 1/2); 
\draw (0, 1/4) -- (4, 1/4);
\draw (2,2) -- (2,0);
\draw (1,1) -- (1,0);
\draw (3,1) -- (3,0);
\draw (1/2, 1/2) -- ( 1/2, 0);
\draw (3/2, 1/2) -- (3/2,0);
\draw (5/2, 1/2) -- (5/2,0);
\draw (7/2, 1/2) -- (7/2, 0);
\draw (1/4, 1/4) -- ( 1/4, 0);
\draw (3/4, 1/4) -- (3/4,0);
\draw (5/4, 1/4) -- (5/4,0);
\draw (7/4, 1/4) -- (7/4, 0);
\draw (9/4, 1/4) -- (9/4, 0);
\draw (11/4, 1/4) -- (11/4,0);
\draw (13/4, 1/4) -- (13/4,0);
\draw (15/4, 1/4) -- (15/4, 0);
\end{tikzpicture}
\caption{Tiling the Carleson box $Q_{I}$ into top-halves $\left\{T_{J}\right\}_{J\in \Dy(I)}$.}
\label{fig2}
\end{figure}
%
%
We will show that $w^\lambda \in B_2$ for any $\lambda >0$. To this end, fix an arbitrary $\la>0$ and let $\Dy(I)$ denote the dyadic decomposition of $I$. Since $w$ has bounded hyperbolic oscillation, it follows that all values of $w$ on top-halves $T_{J} = \{z \in Q_{J} : 1-|z| \geq m(J)/ 2 \}$ of a Carleson square  $Q_{J}$ are comparable to $w(z_{J})$. Observing that the collection $\left\{T_{J}\right\}_{J\in \Dy(I)}$ forms a tiling of $Q_{I}$ (See Figure 1), we can find a constant $C(\la) >0$, such that 
\[ \fr{1}{A(Q_{I})}\int_{Q_{I}} w^{\la} dA = \fr{1}{A(Q_{I})}\sum_{J\in \Dy(I)} \int_{T_{J}} w^{\la} dA \leq \frac{C(\la)}{A(Q_{I})}\sum_{J\in \Dy(I)} A(T_{J})w^{\la}(z_{J}).
\]
Now applying (\ref{thm1est}) to this, we actually get 
\begin{equation*} \fr{1}{A(Q_{I})}\int_{Q_{I}} w^{\la} dA  \leq w^{\la}(z_{I}) \frac{C_1 (\la)}{A(Q_{I})} \sum_{J\in \Dy(I)} A(T_{J}) \left(\fr{m(I)}{m(J)}\right)^{\la \e}. 
\end{equation*}
For $0<\e < 1/ \la$, a straightforward computation shows that
\[  \fr{1}{A(Q_{I})}\sum_{J\in \Dy(I)} A(T_{J}) \left(\fr{m(I)}{m(J)}\right)^{\la \e}  = 
\fr{1}{2}\sum_{n=0}^{\infty} (2^{\la \e -1})^{n} < \infty.
\] 
It now follows that there exists $C_{\la}>0$, such that 
\begin{equation}\label{avest} \fr{1}{A(Q_{I})}\int_{Q_{I}} w^{\la} dA \leq C_{\la} w^{\la}(z_{I}).
\end{equation}
In a similar way, using the leftmost inequality of (\ref{thm1est}), we can also prove an estimate for the dual-weight $w^{-\la}$, identical to that of (\ref{avest}). This is enough to conclude that $w^{\la}$ is a $B_{2}$-weight, for any $\la >0$. 
\end{proof}
%
%
\noindent
Given $\xi \in \partial \mathbb{D}$, let $\Gamma (\xi):= \{z\in \D: |z-\xi|<\sigma(1-|z|) \}$ be the Stolz angle of vertex $\xi$ and fixed aperture $\sigma>1$. Given a function $g \in \B$ and $\varepsilon >0$, we consider the set  
\begin{equation*}
    K(\varepsilon, g) = \{z \in \mathbb{D} : (1-|z|^{2}) |g'(z)| > \varepsilon \}. 
\end{equation*}
In the proof of Theorem \ref{MThm3}, we will use the following description of the closure of $H^p \cap \B$ in $\mathcal{B}$. Fix  $0<p<\infty$ and $g \in \mathcal{B}$. Then $g$ belongs to the closure of $H^p \cap \B$ in $\mathcal{B}$ if and only if, for any $\varepsilon >0$, the area-function 
\begin{equation}\label{A}
    A_{\varepsilon} (g) (\xi) = \left ( \int_{\Gamma(\xi) \cap K(\varepsilon, g)} \frac{dA(z)}{(1-|z|^2)^2} \right)^{1/2} , \quad \xi \in \partial \D,
\end{equation}
belongs to $L^p (\partial \D, dm)$. See \cite{GaMoPa} and \cite{NicGal}.
 We will also use the following description of lacunary series in the Bloch space. A lacunary series is a power series 
 \begin{equation}\label{lacun}
     g(z) = \sum_{k=1}^\infty a_k z^{n_k}, \quad z \in \mathbb{D},
 \end{equation}
with the property that there exists a constant $\rho >1$, such that the positive integers $\{n_k \}$ satisfy  $n_{k+1} \geq \rho \, n_k$, $k=1,2, \ldots$. It was proved in \cite{AnClPo} that a lacunary series of the form \eqref{lacun} belongs to $\mathcal{B}$ if and only if $\sup_{k\geq 1} |a_k| < \infty$. 
\begin{proof}[Proof of Theorem \ref{MThm3}]
Fix a sequence $\{a_k \}_{k=1}^{\infty}$ of complex numbers which does not tend to $0$ and such that $\sup_{k\geq 1} |a_k| \leq 1$. Let $\{n_k \}_{k=1}^{\infty}$ be an increasing sequence of positive integers such that
\begin{equation}\label{sla}
    \lim_{k \to \infty } \frac{n_{k+1}}{n_k} = \infty . \end{equation}
Consider the lacunary series $g$ defined in \eqref{lacun}. For $M>1$ and $k=1,2,\ldots$, consider the annulus
\begin{equation*}
    A_k (M) :=\{z \in \mathbb{D} : \frac{1}{M n_k} \leq 1-|z| \leq \frac{M}{ n_k}  \}. 
\end{equation*}
Fix a point $z \in A_j (2)$ and notice that
\begin{equation*}
    |z||g'(z)| \geq n_j |a_j| |z|^{n_j} - \sum_{k=1}^{j-1} n_k - \sum_{k=j+1}^{\infty} n_k (1- \frac{1}{2n_j})^{n_k}. 
\end{equation*}
Since $1/2 n_j \leq 1-|z| \leq 2/ n_j$, we deduce
\begin{equation}\label{Aj2}
    (1-|z|^{2}) |z||g'(z)| \geq \frac{|a_j|}{2} (1- \frac{2}{n_j})^{n_j} - \frac{4}{n_j} \sum_{k=1}^{j-1} n_k - \frac{4}{n_j} \sum_{k=j+1}^\infty n_k (1- \frac{1}{2n_j})^{n_k}. 
\end{equation}
We claim that the last two terms on the righthand side of \eqref{Aj2} tend to zero, as $j\rightarrow \infty$. Indeed, since $\{n_j\}$ is lacunar, there exists a constant $\rho >1$, such that for all $j=1,2,\ldots$, we have $n_{j+1} \leq \rho (n_{j+1}-n_{j})$. The assumption \eqref{sla} shows that the first sum tends to zero, via
\begin{equation}\label{asymp1} \fr{1}{n_{j}} \sum_{k=1}^{j-1} n_{k} \leq \fr{n_{1}}{n_{j}} + \rho  \fr{(n_{j-1}-n_{1})}{n_{j}} \longrightarrow 0,  \qquad \, j\rightarrow \infty.
\end{equation}
A similar argument also gives
\begin{equation*}\fr{1}{n_{j}} \sum_{k= j+1}^{\infty}n_{k} (1- \frac{1}{2n_j})^{n_k} \leq \fr{n_{j+1}}{n_{j}}(1- \frac{1}{2n_j})^{n_{j+1}}+
  \fr{\rho}{n_{j}} \sum_{k= j+2}^{\infty}(n_{k}-n_{k-1} )(1- \frac{1}{2n_j})^{n_k}
\end{equation*}
\begin{equation}\label{asymp2} \leq \fr{n_{j+1}}{n_{j}}\exp(-n_{j+1}/2n_{j} )+
 \frac{\rho}{n_{j}}   \sum_{k= n_{j+1}}^{\infty}  (1- \frac{1}{2n_j})^{k} \longrightarrow 0,  \qquad  \, j\rightarrow \infty.
 \end{equation}
We have thus proven that the last two terms in (\ref{Aj2}) tend to $0$ as $j \to \infty$. Consequently, given $\e >0$, there exists $j_0 >0$ such that for any $j>j_0$ with $|a_j| > \varepsilon$, we have 
\begin{equation}\label{anell}
    (1-|z|^{2})|z||g'(z)| > C \varepsilon , \quad z \in A_j (2). 
\end{equation}
Here $C>0$ is an absolute constant. Since $\{a_j\}_{j= 1}^{\infty}$ does not tend to zero, there exists $\e_{0}>0$, such that the set $K(\e_{0},g) = \{z \in \mathbb{D} : (1-|z|^{2}) |g'(z)| > \e_{0} \}$ contains infinitely many annulus $A_j (2)$ .
We deduce that the function $A_{\varepsilon} (g)$ defined in \eqref{A} satisfies $A_{\varepsilon} (g) (\xi) = \infty$, for every $\xi \in  \partial \mathbb{D}$. 
Hence the function $g$ does not belong to the closure of $H^p \cap \B$ in $\mathcal{B}$, for any $0<p\leq \infty$.  
\noindent 
We now prove that $g$ satisfies condition (\ref{AMaincond}). Let $M>1$ be a constant to be fixed later and consider
\begin{equation*}
    A(M) = \bigcup_{k=1}^{\infty} A_k (M).
\end{equation*}
Fix $z \in \mathbb{D} \setminus A(M)$ and let $j=j(z)$ be the unique positive integer, satisfying 
\begin{equation*}
    \frac{M}{n_{j+1}} < 1-|z| < \frac{1}{M n_j} .  
\end{equation*}
Write
\begin{equation*}
(1-|z|)|z||g'(z)| \leq \frac{1}{Mn_{j}}\sum_{k=1}^{j} n_k + (1-|z|)\sum_{k = j+1}^{\infty} n_k |z|^{n_k}. 
\end{equation*}
According to (\ref{asymp1}), there exists a constant $C_{1}>1$ independent of $j$, such that 
\[ \frac{1}{n_j} \sum_{k=1}^{j} n_k  \leq C_{1}.
\]
An argument similar to (\ref{asymp2}) shows that there exists a constant $C_{2}>0$, such that
\begin{equation*}
(1-|z|)\sum_{k=j+1 }^{\infty} n_k |z|^{n_k} \leq C_{2}(1-|z|) n_{j + 1} |z|^{n_{j + 1}}. 
\end{equation*}
Now it is straightforward to check that 
\[ \sup_{\frac{M}{n_{j+1}} < 1-|z| < \frac{1}{M n_j} } (1-|z|) n_{j + 1} |z|^{n_{j + 1}} \leq M \left(1-\fr{M}{n_{j+1}} \right)^{n_{j+1}} \leq M e^{-M} .
\]
Combining, we conclude that there exists an absolute constant $C>1$, such that 
\begin{equation*}
    (1-|z|^{2})|z||g'(z)| \leq C \left( \frac{1}{M} + M e^{-M}\right), \qquad  z \in  \D \setminus A(M).
\end{equation*}
Let $0 <\varepsilon < \|g\|_{\B}$ be arbitrary and pick $M=M(\varepsilon) >1$, such that
\begin{equation}\label{petit}
    (1-|z|^{2})|g'(z)| \leq \varepsilon , \quad z \in \mathbb{D} \setminus A(M(\varepsilon)). 
\end{equation}
Denote by $l (\gamma)$ the hyperbolic length of the arc $\gamma \subset \D$, given by 
\begin{equation*}
    l (\gamma) = \int_{\gamma} \frac{|dz|}{1-|z|^{2}}.
\end{equation*}
Note that for any hyperbolic geodesic $\Gamma$ and any $k \geq 1$ we have $l (\Gamma \cap A_k (M(\varepsilon))) \leq 2 \log M(\varepsilon)$. Indeed, the shortest hyperbolic segment joining two concentric circles in $\D$ centered at the origin, is a segment contained in a radius. Thus
\[ l (\Gamma \cap A_k (M(\varepsilon))) \leq \int_{1-M(\e)/n_{k}}^{1-1/M(\e)n_{k}} \fr{dt}{1-t} = 2\log M(\e).
\]
Hence for any hyperbolic segment $\gamma$ of hyperbolic length larger than $K(\varepsilon) := 2 \log M(\varepsilon) / \varepsilon$, we have
\begin{equation}\label{length}
   \frac{ l (\gamma \cap A (M(\varepsilon) )}{l (\gamma )} \leq \varepsilon .
\end{equation}
Let $A_{\e}=A(M(\varepsilon))$ be an abbreviation of that set.
Now let $z, w \in \mathbb{D}$ with $\beta (z,w) > K(\varepsilon)$ and let $\gamma$ be the hyperbolic segment joining $z$ and $w$. Then 
\begin{equation*}
    |g(z) - g(w)| \leq \int_{\gamma} |g' (\xi)| |d \xi|.
\end{equation*}
Using \eqref{length}, we have
\begin{equation}\label{u}
    \int_{\gamma \cap A_{\e}} |g' (\xi)| |d \xi| \leq \|g \|_{\mathcal{B}} l ( \gamma \cap A_{\e}) \leq \varepsilon \|g \|_{\mathcal{B}}  \beta (z,w).
\end{equation}
Applying \eqref{petit}, we also have
\begin{equation}\label{dos}
    \int_{\gamma \cap (\mathbb{D} \setminus A_{\e})} |g' (\xi)| |d \xi| \leq  \varepsilon l ( \gamma) = \varepsilon \beta (z,w).
\end{equation}
Now combining \eqref{u} and \eqref{dos}, it follows that
\begin{equation*}
    |g(z) - g(w)| \leq \varepsilon (\|g\|_{\mathcal{B}} + 1) \beta (z,w), \quad \text{ if } \quad \beta(z,w) \geq K(\varepsilon). 
\end{equation*}
Since $g \in \mathcal{B}$, we automatically have $|g(z) - g(w)| \leq K(\varepsilon) \|g \|_{\mathcal{B}}$, for $\beta (z,w) \leq K(\varepsilon)$. This finishes the proof. 
\end{proof}
%
%
\section{Applications to spectra of Cesar\'o operators.}
In this final section, we shall briefly mention an application which initially sparked the interest in characterizing weights with the property that every power of the weight is in the class $B_{2}$. The background concerns the spectrum of generalized Cesar\'o operators 
\begin{equation}\label{Tgop} T_{g}f (z) = \int_{0}^{z}f(\z)g'(\z) d\z ,  \qquad z\in \D.
\end{equation}
on the Bergman spaces $L^{p}_{a}(\D, dA)$ of analytic functions $f$ in $\D$, which belong to $L^p (\D, dA)$, $p>0$. It is well known that the linear operator $T_{g}$ is bounded on $L^{p}_{a}(\D ,dA)$ if and only if the symbol $g$ belongs to $\B$. For further details on Cesar\'o operators, we refer the reader to \cite{AleCon} and references therein. Now a complete characterization of the spectrum of $T_{g}$ on weighted Bergman spaces has been  given in terms of a $B_{\infty}$-type condition (see \cite{AleCon}, and \cite{AlePottReg} for alternative reformulations). For $L^{p}_{a}(\D, dA)$ with $p>0$, it goes as follows: A complex number $\la \neq 0$ does not belong to the spectrum of $T_{g}$ on $L^{p}_{a}(\D,dA)$ if and only if, there exists a constant $C_{g/\la}>0$, such that the weight $e^{p\Re(g/\la)}$ satisfies
\begin{equation*}\label{Binfcond} \fr{1}{A(Q)}\int_{Q} e^{p\Re(g/\la)} dA \leq C_{g/\la} \exp \left( \fr{1}{A(Q)}\int_{Q} p \Re(g/ \la)dA \right),
\end{equation*} 
for every Carleson square $Q\subset \D$. From this, it follows that $e^{p\Re(g/\la)}$ is a $B_{2}$-weight, if and only if $ \la$ and $- \la$ do not belong to the spectrum of $T_{g}$ on $L^{p}_{a}(\D,dA)$. Actually, if $\| \sigma_{p} (g) \|$ denotes the spectral radius of $T_{g}$ on $L^{p}_{a}(\D, dA)$, then we have  
\begin{equation*}\label{SRB2} \|\sigma_{p}(g) \| =  \inf \{ \la >0: \exp(p\Re(g/\la \xi)) \in B_{2}\, ,\, \forall \xi \in \partial \D \}.
\end{equation*}
From this observation and Lemma \ref{BWconinv}, it follows that the spectral radius is conformally invariant. More precisely, if $g_{z} = (g\circ\phi_{z}) - g(z)$ denotes the hyperbolic translate of $g$, then for any $z\in \D$, we have
\[ \| \sigma_{p}(g) \| = \| \sigma_{p}( g_{z} ) \|.
\] 
We naturally lend the notations of $L^{\infty}(\D)$ and $BMO(\D)$ to include complex-valued functions. As a consequence of \thmref{T1}, we can now successfully give a Bergman space analogue of Theorem 2.4 in \cite{AdBa}.
%
%
\begin{cor}\label{Ap1}There exists an absolute constant $C>0$, such that for any $p>0$ and any $g \in \B$, we have 
\[ \frac{1}{Cp} \, \| \sigma_{p}(g) \| \leq \inf_{h\in L^{\infty}(\D)} \|g-h\|_{BMO(\D)} \leq  \frac{C}{p} \, \| \sigma_{p}(g) \|.
\]
In particular, if the spectrum of $T_{g}$ on $L^{p}_{a}(\D,dA)$ does not contain any non-zero points of the real and imaginary axes, then $g$ belongs to the closure of $L^{\infty}(\D)$ in $BMO(\D)$, and thus the spectrum is $\{0\}$.
\end{cor}
\noindent
A worthy final remark is that from the proof of Corollary \ref{closure}, it follows that the spectral radius is comparable to the infimum of $\e>0$, for which there exists $C(\e)>0$, such that 
\[ |g(z)-g(\z)| \leq C(\e) + \e\beta(z,\z) \qquad, \, z,\z \in\D.
\]
In order to estimate the spectral radius, this serves as a more practical condition. For example, it is evident from the discussions surrounding Corollary \ref{clbl}, that functions which belong to the closure of $H^{\infty}$ in $\B$, induce Cesar\'o operators with zero spectral radius. Meanwhile, \thmref{MThm3} provides a non-trivial example of a Bloch function $g$, such that $T_g$ has zero spectrum, but such that $g$ does not belong to the closure of $H^{p}\cap \B$ in $\B$, for any $0<p \leq \infty$.\\ \\ 
%
%
\noindent \textbf{Acknowledgement.} 
The authors would like to thank Alexandru Aleman and Sandra Pott for fruitful discussions and for their valuable inputs on the content of this manuscript.\\

\bibliographystyle{siam}
\bibliography{MathRef1}

\Addresses
\end{document}